\documentclass{amsart}

\newtheorem{theorem}{Theorem}
\newtheorem{lemma}[theorem]{Lemma}
\newtheorem{proposition}[theorem]{Proposition}
\newtheorem{corollary}[theorem]{Corollary}

\theoremstyle{definition}

\newtheorem{example}[theorem]{Example}
\newtheorem{remark}[theorem]{Remark}

\usepackage{amscd,amssymb}

\begin{document}

\title[Symmetric polynomials in the free metabelian Lie algebras]
{Symmetric polynomials\\ in the free metabelian Lie algebras}

\author[Vesselin Drensky, \c{S}ehmus F\i nd\i k, Nazar \c{S}ah{\.i}n \"O\u{g}\"u\c{s}l\"u]
{Vesselin Drensky, \c{S}ehmus F\i nd\i k, Nazar \c{S}ah{\.i}n \"O\u{g}\"u\c{s}l\"u}
\address{Institute of Mathematics and Informatics,
Bulgarian Academy of Sciences,
1113 Sofia, Bulgaria}
\email{drensky@math.bas.bg}
\address{Department of Mathematics,
\c{C}ukurova University, 01330 Balcal\i,
 Adana, Turkey}
\email{sfindik@cu.edu.tr}
\email{noguslu@cu.edu.tr}

\thanks
{The research of the first named author was partially supported by Grant "Groups and Rings - Theory and Applications"
of the Bulgarian National Science Fund.}

\subjclass[2010]
{17B01, 17B30.}
\keywords{Free metabelian Lie algebras; symmetric polynomials.}

\begin{abstract}
Let $K[X_n]$ be the commutative polynomial algebra in the variables $X_n=\{x_1,\ldots,x_n\}$ over a field $K$ of characteristic zero.
A theorem from undergraduate course of algebra states that the algebra $K[X_n]^{S_n}$ of symmetric polynomials
is generated by the elementary symmetric polynomials which are algebraically independent over $K$.
In the present paper we study a noncommutative and nonassociative analogue of the algebra $K[X_n]^{S_n}$ replacing $K[X_n]$
with the free metabelian Lie algebra $F_n$ of rank $n\geq 2$ over $K$.
It is known that the algebra $F_n^{S_n}$ is not finitely generated
but its ideal $(F_n')^{S_n}$ consisting of the elements of $F_n^{S_n}$ in the commutator ideal $F_n'$ of $F_n$
is a finitely generated $K[X_n]^{S_n}$-module.
In our main result we describe the generators of the $K[X_n]^{S_n}$-module $(F_n')^{S_n}$
which gives the complete description of the algebra $F_n^{S_n}$.
\end{abstract}

\maketitle

\section*{Introduction}
Let $K$ be a field of characteristic zero and let $K[X_n] = K[x_1\ldots,x_n]$ be the commutative associative polynomial algebra generated
by $n$ variables on $K$. Let $KV_n$ be the $n$-dimensional vector space with basis $V_n=\{v_1,\ldots,v_n\}$.
In classical invariant theory one considers the algebra $K[X_n]$ as the algebra of polynomial functions on $KV_n$.
If $f(X_n)=f(x_1,\ldots,x_n)\in K[X_n]$ and $v=\xi_1v_1+\cdots+\xi_nv_n\in KV_n$, $\xi_1,\ldots,\xi_n\in K$,
then $f:v\to f(\xi_1,\ldots,\xi_n)$. The general linear group $GL(KV_n)=GL_n(K)$ acts canonically on the vector space $KV_n$
and this action induces an action on the polynomial functions on $KV_n$:
If $f(X_n)\in K[X_n]$ and $g\in GL_n(K)$, then
\[
g(f):v\to f(g^{-1}(v)),\quad v\in KV_n.
\]
If $G$ is a subgroup of $GL_n(K)$ then the algebra
\[
K[X_n]^G=\{f(X_n)\in K[X_n]\mid g(f)=f\text{ for all }g\in G\}
\]
is the algebra of $G$-invariants.

A special case and the main motivation of Hilbert's fourteenth problem \cite{H} is whether the algebra of $G$-invariants $K[X_n]^G$
is finitely generated for any subgroup $G$ of $GL_n(K)$.
This problem was solved into affirmative for finite groups by Emmy Noether \cite{N}.
In particular, for the symmetric group $S_n$ of degree $n$ with its action by permuting of the variables $X_n$
the description of the algebra $K[X_n]^{S_n}$ of symmetric polynomials over a field $K$ of arbitrary characteristic is given in the following
theorem from undergraduate course of algebra known as Fundamental Theorem on Symmetric Polynomials.

\begin{theorem}\label{symmetric polynomials}
The algebra $K[X_n]^{S_n}$
is generated by the elementary symmetric polynomials
\[
e_j=\sum_{i_1<\cdots<i_j}x_{i_1}\cdots x_{i_j},\quad j=1,\ldots,n.
\]
which are algebraically independent over $K$.
\end{theorem}

One of the branches of noncommutative invariant theory studies the invariants of $G< GL_n(K)$ when $GL_n(K)$ acts
on the free associative algebra $K\langle X_n\rangle=K\langle x_1,\ldots,x_n\rangle$, the free Lie algebra $L_n=L(X_n)$,
and on their homomorphic images modulo ideals invariant under the action of $GL_n(K)$. Replacing the group $GL_n(K)$
with its opposite, in this case it is more convenient to assume that $GL_n(K)$ acts canonically on the vector space $KX_n$ with basis $X_n$
instead on the vector space $KV_n$.

The algebra $K\langle X_n\rangle^{S_n}$ was considered first by Wolf \cite{Wo} in 1936, see \cite{BRRZ} and \cite{GKL}
for a survey and further development.
The result of Dicks and Formanek \cite{DiFo} and Kharchenko \cite{Kh} gives
that the algebra of invariants $K\langle X_n\rangle^G$ is finitely generated if and only if $G$ is a
cyclic group acting by scalar multiplication. On the other hand Koryukin \cite{Ko} studied $K\langle X_n\rangle^G$
with the additional action of the symmetric group $S_d$ permuting the positions of the variables in the homogeneous component
of degree $d$ of $K\langle X_n\rangle$:
\[
\left(\sum\alpha_ix_{i_1}\cdots x_{i_d}\right){\sigma}=\sum\alpha_ix_{i_{\sigma(1)}}\cdots x_{i_{\sigma(d)}},\quad \sigma\in S_d.
\]
He showed that under this additional action the algebra $K\langle X_n\rangle^G$ is finitely generated for reductive groups $G$.

Another popular topic in noncommutative invariant theory is the study of invariants of groups acting on relatively free algebras
in varieties of associative algebras. See, e.g., the recent paper \cite{DoDr} and the references there for the common features
and the differences with classical invariant theory.

When a nontrivial finite group $G$ acts on the free Lie algebra $L_n$ by a result of Bryant \cite{Br}
the algebra of invariants $L_n^G$ is never finitely generated.
A crucial role is played by the free metabelian Lie algebra $F_n=F_n({\mathfrak A}^2)=L_n/L_n''$, $L_n''=[[L_n,L_n],[L_n,L_n]]$,
which is the relatively free algebra in the variety ${\mathfrak A}^2$ of metabelian (solvable of class two) Lie algebras.
By the theorem of Zelmanov \cite{Z} if a Lie algebra satisfies the Engel identity then it is nilpotent.
As a consequence there is a dichotomy for varieties of Lie algebras.
The variety is either nilpotent or all of its polynomial identities follow from the metabelian identity $[[y_1,y_2],[y_3,y_4]]=0$.
As in the case of free Lie algebras, by a results of Drensky \cite{Dr} the algebra
$F_n({\mathfrak V})^G$ is never finitely generated for nontrivial finite groups $G$
and varieties $\mathfrak V$ containing the metabelian variety.
Although the algebra $F_n^G=F_n({\mathfrak A}^2)^G$ is not finitely generated it satisfies the property
that its ideal $(F_n')^G$ consisting of the elements of $F_n^G$ in the commutator ideal $F_n'$ of $F_n$
is a finitely generated $K[X_n]^G$-module under an appropriate action of $K[X_n]$ on $F_n'$.

In the present paper we describe the algebra $F_n^{S_n}$
of the symmetric polynomials in the free metabelian Lie algebra $F_n$, $n\geq 2$,
and give a generating set of the $K[X_n]^{S_n}$-module $(F_n')^{S_n}$. The case $n=2$ was handled by
F{\i}nd{\i}k and \"O{\v g}\"u{\c s}l\"u \cite{FO}.
They gave the description of the symmetric polynomials and an infinite generating set for the algebra $F_2^{S_2}$.

\begin{theorem}\label{the case n=2}
Let $K$ be a field of characteristic zero. Then the symmetric polynomials in the free two-generated metabelian Lie algebra $F_2$
are of the form
\[
f(X_2)=\alpha(x_1+x_2)
+\sum_{0\leq a<b}\alpha_{ab}[x_2,x_1](\text{\rm ad}^ax_1\text{\rm ad}^bx_2-\text{\rm ad}^ax_2\text{\rm ad}^bx_1),
\quad \alpha,\alpha_{ab}\in K.
\]
\end{theorem}

As a consequence, one immediately obtains the following.

\begin{corollary}\label{finite generation for n=2}
The $K[X_2]^{S_2}$-module $(F_2')^{S_2}$ is generated by
\begin{equation}\label{generator for n=2}
f_{12}(X_2)=[x_2,x_1](\text{\rm ad}x_2-\text{\rm ad}x_1).
\end{equation}
\end{corollary}

\section{Preliminaries}\label{first section}

In the sequel, we fix the field $K$ of characteristic zero,
the set $X_n$ consisting of the variables $x_1,\ldots,x_n$, $n\geq 2$,
and denote by $F_n$ the free metabelian Lie algebra generated by the set $X_n$ over the base field $K$.
It is well known that the commutator ideal $F_n'=[F_n,F_n]$ of $F_n$ has a basis consisting of the elements of the form
\[
[\ldots[[x_{i_1},x_{i_2}],x_{i_3}],\ldots,x_{i_d}]=[x_{i_1},x_{i_2},x_{i_3},\ldots,x_{i_d}]
=[x_{i_1},x_{i_2}]\text{ad}x_{i_3}\cdots \text{ad}x_{i_d},
\]
where $i_1 > i_2 \leq i_3\leq\cdots\leq i_d$. The identity
\[
[x_{i_1},x_{i_2}]\text{ad}x_{i_3}\cdots \text{ad}x_{i_d}=[x_{i_1},x_{i_2}]\text{ad}x_{i_{\tau(3)}}\cdots \text{ad}x_{i_{\tau(d)}}
\]
for any permutation $\tau$ of $\{3,\ldots,d\}$ allows to equip $F_n'$ with a $K[X_n]$-module structure by
\[
fp(x_1,\ldots,x_n)=fp(\text{ad}x_1,\ldots,\text{ad}x_n),\quad f\in F_n',p(x_1,\ldots,x_n)\in K[X_n].
\]
For more details we refer to the book by Bahturin \cite{Ba} on the theory of Lie algebras and their identities.

We make use of the embedding of the free metabelian Lie algebra $F_n$ into an abelian wreath product
of Lie algebras due to Shmel'kin \cite{Sh}. Let $KU_n$ and $KV_n$ be the abelian Lie algebras with bases
$U_n=\{u_1,\ldots,u_n\}$ and $V_n=\{v_1,\ldots,v_n\}$, respectively.
Let $W_n$ be the free right $K[X_n]$-module with free generators $u_1,\ldots,u_n$.
We assume that $W_n$ is a Lie algebra with trivial multiplication. The abelian wreath product $(KU_n)\text{wr}(KV_n)$
is equal to the semidirect sum $W_n\leftthreetimes KV_n$. The elements of $W_n\leftthreetimes KV_n$ are of the form
\begin{equation}\label{element of wreath product}
w=\sum_{i=1}^nu_ip_i(X_n)+\sum_{i=1}^n\alpha_iv_i,\quad p_i(X_n)\in K[X_n],\alpha_i\in K.
\end{equation}
The multiplication in $(KU_n)\text{wr}(KV_n)$ is defined by
\[
[W_n,W_n] = [V_n,V_n] = 0,\quad [u_ip_i(X_n),v_j]=u_ip_i(X_n)x_j,\quad i,j=1,\ldots,n.
\]
Thus the the abelian wreath product $(KU_n)\text{wr}(KV_n)$ is a metabelian Lie algebra. Since the Lie algebra $F_n$ is free
in the variety ${\mathfrak A}^2$ of all metabelian Lie algebras,
then every mapping $X_n\to (KU_n)\text{wr}(KV_n)$ can be extended to a homomorphism $F_n\to (KU_n)\text{wr}(KV_n)$. The homomorphism
$\delta:F_n\to (KU_n)\text{wr}(KV_n)$ defined by
\[
\delta: x_i\to u_i + v_i,\quad i = 1,\ldots,n,
\]
is a monomorphism as a special case of the embedding theorem of Shmel'kin \cite{Sh}.
In the sequel we shall identify the elements of $F_n$ and their images under $\delta$ in $(KU_n)\text{wr}(KV_n)$. If
\[
f=\sum_{i>j}[x_i,x_j]p_{ij}(X_n)\in F_n',\quad p_{ij}(X_n)\in K[X_n],
\]
then
\[
\delta(f)=\sum_{i>j}(u_ix_j-u_jx_i)p_{ij}(X_n)\in W_n\subset (KU_n)\text{wr}(KV_n).
\]
On the other hand, the element $w$ from (\ref{element of wreath product})
belongs to the image of the commutator ideal $F_n'$ of $F_n$
if and only if $\alpha_i=0$, $i=1,\ldots,n$, and
\begin{equation}\label{belonging to F'}
\sum_{i=1}^nx_ip_i(X_n)=0.
\end{equation}
The general linear group $GL_n(K)$ acts simultaneously on the vector spaces $KU_n$ and $KV_n$
in the same way as on the vector space $KX_n$ with basis $X_n$. This action is extended canonically on
the abelian wreath product $(KU_n)\text{wr}(KV_n)$. The action of $GL_n(K)$ on $(KU_n)\text{wr}(KV_n)$ is the same as the action on the factor algebra
$K[U_n,X_n]/I$ of the polynomial algebra $K[U_n,X_n]$
modulo the ideal $I$ generated by $U_n^2=\{u_iu_j\mid 1\leq i\leq j\leq n\}$.
Hence, as in \cite[Section 3]{DrF}, $w\in (KU_n)\text{wr}(KV_n)$ from (\ref{element of wreath product}) can be identified with the element
\begin{equation}\label{image in polynomial algebra in 2n variables}
\pi(w)=\sum_{i=1}^nu_if_i(X_n)+\sum_{i=1}^n\alpha_ix_i\in K[U_n,X_n].
\end{equation}
For a subgroup $G$ of $GL_n(K)$ we can consider the $G$-invariants $((KU_n)\text{wr}(KV_n))^G$.
The following easy assertion gives that for a large class of groups $G$ the vector space $W_n^G\subset((KU_n)\text{wr}(KV_n))^G$
is a finitely generated $K[X_n]^G$-module. The leading idea of the proof is to find an object which has nice properties from the point
of view of classical invariant theory and then to transfer these properties to $W_n^G$.
A similar idea was used in the proof of \cite[Proposition 2.2]{DoDr}.

\begin{lemma}\label{finite generation for class of groups}
Let $G$ be a subgroup of $GL_n(K)$ such that the algebra of invariants $K[U_n,X_n]^G$ is finitely generated. Then
$W_n^G$ is a finitely generated $K[X_n]^G$-module.
\end{lemma}

\begin{proof}
Let the algebra $K[U_n,X_n]^G$ be generated by $h_j(U_n,X_n)$, $j=1,\ldots,m$.
Since $GL_n(K)(KU_n)=KU_n$ and $GL_n(K)(KX_n)=KX_n$, the homogeneous components with respect to $U_n$ and to $X_n$
are stable under the action of $GL_n(K)$. Hence without loss of generality we may assume
that $h_j(U_n,X_n)$ are homogeneous with respect to $U_n$ and to $X_n$.
Then the algebra $K[X_n]^G$ is generated by those $h_j=h_j(X_n)$ which do not depend on $U_n$ and
as a $K[X_n]^G$-module $W_n^G$ is generated by those $h_j(U_n,X_n)$ which are linear with respect to $U_n$.
\end{proof}

\begin{corollary}\label{finite generated module}
If $G<GL_n(K)$ is a finite group then $(F_n')^G$ is a finitely generated $K[X_n]^G$-module.
\end{corollary}

\begin{proof}
As a consequence of the theorem of Emmy Noether \cite{N} we obtain that $W_n^G$ is a finitely generated $K[X_n]^G$-module.
Since the algebra $K[X_n]^G$ is noetherian and $(F_n')^G$ is a $K[X_n]^G$-submodule of $W_n^G$, we obtain immediately that
$(F_n')^G$ is also finite generated as a $K[X_n]^G$-module.
\end{proof}

The following proposition is a partial case of the well known description of the generators of the algebra of $S_n$-invariants
$K[X_n,\ldots,Y_n]^{S_n}$ where the symmetric group acts simultaneously on several sets of variables $X_n,\ldots,Y_n$.
For the proof see e.g. \cite[Chapter 2]{We}.

\begin{proposition}\label{Sym acting on several sets}
The algebra $K[U_n,X_n]^{S_n}$ is generated by the following polynomials which are polarizations of the elementary symmetric polynomials
\begin{equation}\label{polarizations}
e_{p,q}(U_n,X_n)=\sum u_{i_1}\cdots u_{i_p}x_{j_1}\cdots x_{j_q}, \quad 0<p+q\leq n,
\end{equation}
where the summation runs on all tuples $(i_1,\ldots,i_p,j_1,\ldots,j_q)$ consisting of pairwise different numbers
and such that $1\leq i_1<\cdots <i_p\leq n$, $1\leq j_1<\cdots<j_q\leq n$.
\end{proposition}

As a combination of Proposition \ref{Sym acting on several sets} and the proof of Lemma \ref{finite generation for class of groups}
we obtain the following.

\begin{corollary}\label{Wn as module}
The $K[X_n]^{S_n}$-module $W_n^{S_n}$ is generated by the polynomials from (\ref{polarizations})
\[
e_{1,q}(U_n,X_n)=\sum_{i=1}^n\sum_{j_1<\cdots<j_q}u_ix_{j_1}\cdots x_{j_q},\quad i\not=j_1,\ldots,j_q,q=0,1,\ldots,n-1.
\]
\end{corollary}

\section{Main results}

In this section we shall find a minimal system of generators of the $K[X_n]^{S_n}$-module $(F_n')^{S_n}$.
As a consequence we shall find a system of generators of the Lie algebra $F_n^{S_n}$.
Applying (\ref{belonging to F'}) to Corollary \ref{Wn as module}
the generators of the $K[X_n]^{S_n}$-module $(F_n')^{S_n}$
are of the form
\begin{equation}\label{symmetric in F'}
h_p(U_n,X_n)=\sum_{q=0}^{n-1}e_{1,q}(U_n,X_n)p_q(X_n),\quad p_q(X_n)\in K[X_n]^{S_n},
\end{equation}
with the property $h_p(X_n,X_n)=0$. In the proof of our main result we shall need the following easy lemma.

\begin{lemma}\label{group of solutions of linear equation}
Consider the vector space $Z\subset K^n$ consisting of all solutions $t=(t_1,\ldots,t_n)$
of the equation
\begin{equation}\label{linear equation}
\sum_{j=1}^njt_j=0.
\end{equation}
If $c=(c_1,\ldots,c_n)\in Z$ has nonzero coordinates at positions $j_1<\cdots<j_m$ only, then
$c$ is a linear combination of
\[
z_{j_1j_k}=(0,\ldots,0,j_k,0,\ldots,0,-j_1,0,\ldots,0),\quad k=2,\ldots,m,
\]
where the nonzero coordinates of $z_{j_1j_k}$ are in the $j_1$-th and $j_k$-th positions.
\end{lemma}

\begin{proof}
Obviously all $z_{j_1j_k}$ satisfy the equation (\ref{linear equation}).
Let $c=(c_1,\ldots,c_n)$ be an arbitrary solution of (\ref{linear equation})
satisfying the restrictions on its nozero coordinates.
Then
\[
\frac{1}{j_1}\left(j_1c+\sum_{k=2}^mc_jz_{j_1j_k}\right)=\frac{1}{j_1}\left(j_1c_{j_1}+\sum_{k=2}^mj_kc_{j_k},0,\ldots,0\right)=(0,0,\ldots,0)
\]
and $c$ can be expressed in terms of $z_{j_1j_k}$, $k=2,\ldots,m$.
\end{proof}

In the sequel we shall denote
\[
\varepsilon_{q+1}(U_n,X_n)=e_{1,q}(U_n,X_n)=\sum_{i=1}^nu_i\sum x_{j_1}\cdots x_{j_q},\quad 1\leq j_1<\cdots<j_q\leq n,i\not=j_l,
\]
\[
e_q(X_n)=e_{0,q}(X_n)=\sum x_{j_1}\cdots x_{j_q},\quad 1\leq j_1<\cdots<j_q\leq n.
\]
In this notation Corollary \ref{Wn as module} gives that the $K[X_n]^{S_n}$-module $W_n^{S_n}$ is generated by the elements
\[
\varepsilon_1(U_n,X_n)=\varepsilon_1(U_n),\varepsilon_2(U_n,X_n),\ldots,\varepsilon_n(U_n,X_n),
\]
and the equation (\ref{symmetric in F'}) becomes
\begin{equation}\label{symmetric in F'-2}
h_r(U_n,X_n)=\sum_{j=1}^n\varepsilon_j(U_n,X_n)r_j(X_n),\quad r_j(X_n)\in K[X_n]^{S_n},
\end{equation}

The following theorem gives a generating set of $(F_n')^{S_n}$
as a $K[X_n]^{S_n}$-module.

\begin{theorem}\label{the main result}
Identifying the elements of $F_n$ with their images in the abelian wreath product $(KU_n)\text{\rm wr}(KV_n)$,
as a $K[X_n]^{S_n}$-module $(F_n')^{S_n}$ is generated by the polynomials
\begin{equation}\label{u(ij)}
h_{ij}(U_n,X_n)=j\varepsilon_i(U_n,X_n)e_j(X_n)-i\varepsilon_j(U_n,X_n)e_i(X_n),\quad 1\leq i<j\leq n.
\end{equation}
\end{theorem}

\begin{proof}
Let $h_r(U_n,X_n)\in W_n^{S_n}$ be a polynomial in the form (\ref{symmetric in F'-2}) satisfying the property $h_r(X_n,X_n)=0$.
We write the symmetric polynomials $r_j(X_n)$ as
\[
r_j(X_n)=\sum_{aj}\alpha_{aj}e_1^{a_1}(X_n)\cdots e_j^{a_j-1}(X_n)\cdots e_n^{a_n}(X_n),\quad a_i\geq 0,a_j\geq 1,\alpha_a\in K.
\]
Since $\varepsilon_j(X_n,X_n)=je_j(X_n)$ the condition $h_r(X_n,X_n)=0$ implies
\[
\sum_{j=1}^n\sum_{aj}j\alpha_{aj}e_1^{a_1}(X_n)\cdots e_j^{a_j}(X_n)\cdots e_n^{a_n}(X_n)=0.
\]
Since $e_1(X_n),\ldots,e_n(X_n)$ are algebraically independent we obtain that
\[
\left(\sum_{j=1}^nj\alpha_{aj}\right)e_1^{a_1}(X_n)\cdots e_j^{a_j}(X_n)\cdots e_n^{a_n}(X_n)=0
\]
and
\begin{equation}\label{fixed powers of e_i}
\sum_{j=1}^nj\alpha_{aj}=0
\end{equation}
for every fixed $a=(a_1,\ldots,a_n)$.
Since the degree of $e_j(X_n)$ in $r_j(X_n)$ is equal to $a_j-1$, the coefficient $\alpha_{aj}$
in (\ref{fixed powers of e_i}) is equal to 0 if $a_j=0$. Now we fix $a=(a_1,\ldots,a_n)\not=(0,\ldots,0)$
and consider the summand
\[
h^{(a)}(U_n,X_n)=\sum_{j=1}^n\alpha_{aj}\varepsilon_j(U_n,X_n)e_1^{a_1}(X_n)\cdots e_j^{a_j-1}(X_n)\cdots e_n^{a_n}(X_n)
\]
of $h_r(U_n,X_n)$ in (\ref{symmetric in F'-2}). The condition (\ref{fixed powers of e_i}) guarantees that at least two
$a_i$ and $a_j$ are different from 0. Let the nonzero coefficients $\alpha_{aj}$ be with indices $j_1<\cdots<j_m$.
Since the corresponding summand has $\varepsilon_j(U_n,X_n)$ as a factor, we derive that $a_{j_k}>0$ for all $k=1,\ldots,m$.
Since $a=(a_1,\ldots,a_n)$ satisfies the equation (\ref{fixed powers of e_i}), by Lemma \ref{group of solutions of linear equation}
we can express it in the form
\[
a=\sum_{k=2}^m\beta_kz_{j_1j_k},\quad \beta_k\in K.
\]
Then we rewrite $h^{(a)}(U_n,X_n)$ in the form
\[
h^{(a)}(U_n,X_n)=\sum_{k=2}^m\beta_kh_{j_1j_k}(U_n,X_n)e_1^{a_1}(X_n)\cdots e_{j_1}^{a_{j_1}-1}(X_n)\cdots e_{j_k}^{a_{j_k}-1}(X_n)\cdots e_n^{a_n}(X_n).
\]
Hence $h^{(a)}(U_n,X_n)$ belongs to the $K[X_n]^{S_n}$-module generated by $h_{ij}(U_n,X_n)$, $1\leq i<j\leq n$.
\end{proof}

\begin{remark}
It is interesting to find a presentation (in terms of generators and defining relations) of
the $K[X_n]^{S_n}$-module $(F_n')^{S_n}$.
It is easy to check that for $1\leq i<j<k\leq n$ the generators (\ref{u(ij)}) from Thereom \ref{the main result} satisfy the relation
\[
kh_{ij}(U_n,X_n)e_k(X_n)-jh_{ik}(U_n,X_n)e_j(X_n)+ih_{jk}(U_n,X_n)e_i(X_n)=0,
\]
but it is not clear whether they are enough to determine $(F_n')^{S_n}$ as a factor-module
of the free $K[X_n]^{S_n}$-module freely generated by the elements (\ref{u(ij)}).
\end{remark}

\begin{corollary}
The Lie algebra $F_n^{S_n}$ is generated by
\[
f_1(X_n)=\varepsilon_1(U_n,X_n)=x_1+\cdots+x_n
\]
(expressed as an element of $F_n$) and
\[
h_{ij}(U_n,X_n)e_2^{a_2}(X_n)\cdots e_n^{a_n}(X_n),\quad a_2,\ldots,a_n\geq 0,
\]
(expressed as an element in the image of $F_n$ in $(KU_n)\text{\rm wr}(KV_n)$).
\end{corollary}

\begin{proof}
Up to a multiplicative constant the polynomial $f_1(X_n)=x_1+\cdots+x_n\in F_n^{S_n}$ is the only nonzero element
in $F_n^{S_n}/(F_n')^{S_n}$. Hence for the proof of the corollary it is sufficient to consider the elements in $(F_n')^{S_n}$ only.
By Theorem \ref{the main result} as a vector space $(F_n')^{S_n}$ is spanned by the polynomials $f_{ij}^{(a)}(X_n)\in (F_n')^{S_n}$
which correspond to the elements
\[
h_{ij}^{(a)}=h_{ij}(U_n,X_n)\prod_{k=1}^ne_k^{a_k}(X_n)\in W_n\subset (KU_n)\text{\rm wr}(KV_n).
\]
The commuting of $f_{ij}^{(a)}(X_n)$ by $f_1(X_n)=x_1+\cdots+x_n$ multiplies by $e_1(X_n)$ its image $h_{ij}^{(a)}$ in $W_n$.
Hence, starting with $f_{ij}^{(b)}(X_n)$ for $b=(0,a_2,\ldots,a_n)$ and commuting several times with
$f_1(X_n)$ we can obtain $f_{ij}^{(a)}(X_n)$ for $a=(a_1,a_2,\ldots,a_n)$. Since $[F_n',F_n']=0$,
we need all $f_{ij}^{(b)}(X_n)$, $b=(0,a_2,\ldots,a_n)$, to produce the basis elements $f_{ij}^{(a)}(X_n)$ of $(F_n')^{S_n}$.
\end{proof}

\begin{example}
The generators of the $K[X_n]^{S_n}$-module $(F_n')^{S_n}$ in Theorem \ref{the main result} are given as elements
of the ideal $W_n$ of the abelian wreath product $(KU_n)\text{wr}(KV_n)$.
Now we shall give their explicit form as elements of $(F_n')^{S_n}$ for $n=2$ and 3.

Let $n=2$. According to Theorem \ref{the main result} the $K[X_2]^{S_2}$-module $(F_2')^{S_2}$ is generated by
\[
h_{12}(U_2,X_2)=2\varepsilon_1(U_2,X_2)e_2(X_2)-\varepsilon_2(U_2,X_2)e_1(X_2)
\]
\[
=2(u_1+u_2)(x_1x_2)-(u_1x_2+u_2x_1)(x_1+x_2)=(u_1x_2-u_2x_1)(x_1-x_2)
\]
which corresponds to the generator
\[
f_{12}(X_2)=[x_2,x_1,x_2-x_1]\in (F_2')^{S_2}
\]
from Corollary \ref{finite generation for n=2}.

Let $n=3$. Then $K[X_3]^{S_3}$-module $(F_3')^{S_3}$ is generated by the elements
$f_{12}(X_3)$, $f_{13}(X_3)$, $f_{23}(X_3)$ corresponding, respectively, to
$h_{12}(U_3,X_3)$, $h_{13}(U_3,X_3)$, $h_{23}(U_3,X_3)$.
Direct computations show that
\[
f_{12}(X_3)=[x_2,x_1,x_2-x_1]+[x_3,x_1,x_3-x_1]+[x_3,x_2,x_3-x_2],
\]
\[
f_{13}(X_3)=[x_2,x_1,x_2-x_1,x_3]+[x_3,x_1,x_3-x_1,x_2]+[x_3,x_2,x_3-x_2,x_1],
\]
\[
f_{23}(X_3)=[x_2,x_1,x_2-x_1,x_1+x_2,x_3]
\]
\[
+[x_3,x_1,x_3-x_1,x_1+x_3,x_2]+[x_3,x_2,x_3-x_2,x_2+x_3,x_1].
\]
With some additional efforts we can present all commutators as linear combinations of commutators in the form
$[x_{i_1},x_{i_2},x_{i_3}]$, $i_1>i_2\leq i_3$, and $[x_{i_1},x_{i_2},x_{i_3},x_{i_4}]$, $i_1>i_2\leq i_3\leq i_4$.
\end{example}

\end{document}